\documentclass[11pt]{amsart}

\usepackage{mathtools,amsmath,amssymb}
\usepackage[usenames,dvipsnames]{color}
\usepackage[hyphens]{url}
\usepackage[pagebackref,linktocpage=true,colorlinks=true,linkcolor=RoyalBlue,citecolor=BrickRed,urlcolor=RoyalBlue]{hyperref}
\usepackage[msc-links,abbrev]{amsrefs}

\theoremstyle{plain}
\newtheorem{thm}{Theorem}[section]

\newtheorem{lem}[thm]{Lemma}

\theoremstyle{definition}
\newtheorem{note}[thm]{Note}

\newcommand{\R}{\mathbf{R}}
\newcommand{\Hyp}{\mathbf{H}}
\newcommand{\ol}{\overline}
\newcommand{\C}{\mathcal{C}}
\newcommand{\M}{\mathcal{M}}

\newcommand{\sff}{\mathrm{I\kern-0.09emI}}

\renewcommand{\tilde}{\widetilde}
\renewcommand{\phi}{\varphi}
\renewcommand{\epsilon}{\varepsilon}
\renewcommand{\leq}{\leqslant}
\renewcommand{\geq}{\geqslant}

\DeclareMathOperator{\I}{I}

\begin{document}

\vspace*{-0.5in}
	
\title[A local isoperimetric inequality]{A local isoperimetric inequality\\ for balls with nonpositive curvature}

\begin{abstract}
We show that small perturbations of the metric of a ball in Euclidean $n$-space to metrics with nonpositive curvature do not reduce the isoperimetric ratio. Furthermore, the isoperimetric ratio is preserved only if the perturbation corresponds to a homothety of the ball. These results establish a sharp local version of the Cartan-Hadamard conjecture.
\end{abstract}

\author{Mohammad Ghomi}
\address{School of Mathematics, Georgia Institute of Technology, Atlanta, Georgia 30332}
\email{ghomi@math.gatech.edu}
\urladdr{https://ghomi.math.gatech.edu}

\author{John Ioannis Stavroulakis}
\address{School of Mathematics, Georgia Institute of Technology, Atlanta, Georgia 30332}
\email{jstavroulakis3@gatech.edu, john.ioannis.stavroulakis@gmail.com}

\subjclass[2020]{Primary 53C20, 58J05; Secondary 49Q2, 52A38}
\date{Last revised on \today}
\keywords{Isoperimetric ratio, metrics of nonpositive curvature, Rauch comparison theorem, Cartan-Hadamard conjecture, $\textup{CAT}(0)$ spaces, Hyperbolic space.}
\thanks{The first-named author was supported by NSF grant DMS-2202337}

\maketitle

\section{Introduction}

The \emph{isoperimetric ratio} of a compact Riemannian $n$-manifold $\Omega$ with boundary $\partial\Omega$ is given by 
$
\I(\Omega)\coloneqq   |\partial\Omega|^n/|\Omega|^{n-1},
$
where $|\Omega|$ denotes the volume and $|\partial\Omega|$ the perimeter of $\Omega$.
 The \emph{Cartan-Hadamard conjecture} \cite{aubin1975,gromov1981,burago-zalgaller1980} states that if $\Omega$ forms a domain in a complete simply connected manifold of nonpositive (sectional) curvature, known as a \emph{Cartan-Hadamard manifold}, then 
 $$
 \I(\Omega)\geq \I(B^n_\delta),
 $$
  where $B^n_{\delta}:=(B^n,\delta)$ is the closed unit  ball $B^n\subset\R^n$ endowed with the Euclidean metric $\delta$. Furthermore, $\I(\Omega)=\I(B^n_\delta)$ only if $\Omega$ is isometric to a Euclidean ball.
  Here we show that the conjecture holds in a local sense.
 Let $\M _{0}(B^n)$ be the space of  $\C^\infty$ nonpositively curved metrics $g=(g_{ij})$ on $B^n$ with the $\C^2$-norm 
  $
 |g|_{\C^2(B^n)}\coloneqq   \sup_{ij}|g_{ij}|_{\C^2(B^n)},
  $
and $B^n_g$ denote the corresponding Riemannian manifolds.
    
  \begin{thm}\label{thm:main}
There exists $\epsilon>0$ such that  for all metrics $g\in \M  _{0}(B^n)$ with $|g-\delta|_{\C^2(B^n)}\leq\epsilon$,  
$
\I(B^n_g)\geq \I(B^n_\delta).
$
Furthermore,  $\I(B^n_g)= \I(B^n_\delta)$ only if $B^n_g$ is isometric to a Euclidean ball.
  \end{thm}
  
  To establish this result we show that, for small $\epsilon$,  $B^n_g$ is isometric via  normal coordinates to a star-shaped domain $\Omega\subset \R^n$ with metric $\ol g$,  denoted $\Omega_{\ol g}:=(\Omega,\ol g)$. So  $\partial\Omega_{\ol g}$ is the graph of a radial function $f$ on the unit sphere $S^{n-1}$ in the Euclidean sense.  Using Rauch's comparison theorem, we find a general inequality \eqref{eq:main} for $\I(\Omega_{\ol g})$ in terms of $f$ and the Jacobian of the exponential map of $\Omega_{\ol g}$  which may be of independent interest. It follows that $\I(\Omega_{\ol g})\geq \I(\Omega_{\delta})$ for small $\epsilon$, via a variational technique that we devise below. But $\I(\Omega_{\delta})\geq \I(B^n_\delta)$, by the classical isoperimetric inequality in $\R^n$, which completes the proof. Refining this method, we generalize Theorem \ref{thm:main} to metrics with curvature $\leq k\leq 0$,  as described in Theorem \ref{thm:main2}.
 
 The Cartan-Hadamard conjecture, which would extend the classical isoperimetric inequality \cite{osserman1978, chavel2001},  is known to hold only in dimensions $\leq 4$ \cite{weil1926,kleiner1992,croke1984}. Some partial results are also known in higher dimensions for geodesic balls \cite{ballmann-gromov-schroeder}, small volumes \cite{morgan-johnson2000,nardulli-acevedo2020,druet2002}, large volumes \cite{yau1975,burago-zalgaller1988},  hyperbolic space \cite{ritore2023,borbely2002}, or with a dimension-dependent constant  \cite{hoffman-spruck1974,croke1984}.  In contrast to Theorem \ref{thm:main}, the results for small and large volumes all require a negative upper bound for curvature. See \cite{kloeckner-kuperberg2019,ghomi-spruck2022,ghomi-spruck2023b,li-lin-xu2025} for some more recent studies, \cite{druet2010} for an introduction to the problem, and
 \cite{druet-hebey2002} for applications to Sobolev inequalities.

\begin{note}
The condition in Theorem \ref{thm:main} that the metric $g$ be close to the Euclidean metric $\delta$ cannot be omitted. Indeed there exist negatively curved metrics $g$ on the unit ball $B^3$ such that $I(B^3_g)$ is arbitrarily small \cite[Cor. 1.10]{kloeckner-kuperberg2019}. These examples also show that nonpositively curved balls cannot in general be embedded isometrically in a Cartan-Hadamard manifold (since the Cartan-Hadamard conjecture holds in dimension $3$). Thus while Theorem \ref{thm:main} adds more credence to the Cartan-Hadamard conjecture, it will not be subsumed by a positive resolution of it.
\end{note}

 \begin{note}\label{note:H}
Theorem \ref{thm:main} does not generalize  to geodesic balls in the hyperbolic space $\Hyp^n_k$, of constant curvature $k<0$, since the isoperimetric ratio of balls in $\Hyp^n_k$ is an increasing function of their radius. More precisely, consider the geodesic balls $U(r)$ of radius $r\leq 1$ centered at a point $o$ of $\Hyp^n_k$. There are natural diffeomorphisms $\phi_r\colon B^n\to U(r)$ given by rescalings of normal coordinates centered at $o$. Let $g_r$ be the corresponding pullback metrics. Then $g_r$ has constant curvature $k$, and $g_r\to g_1$ in the $\C^2$-topology, as $r\to 1^{-}$, but $\I(B^n_{g_r})< \I(B^n_{g_1})$ for $r<1$.
\end{note}
 
  \section{Preliminaries}

Let $\M_k(B^n)$ be the space of $\C^\infty$ metrics $g$ on $B^n$ with curvature bounded above by  $k\in(-\infty,0]$. For each $x\in B^n$ we identify $g_x$ with its matrix representation $(g_{ij}(x))$ with respect to the standard basis $e_i$ of $\R^n$. So 
\begin{equation}\label{eq:g}
g_x(v,w)=v^T \! g_x  \,\! w,
\end{equation}
for tangent vectors $v$, $w\in T_xB^n_g\simeq\R^n$.
The \emph{$\C^\ell$-topology} on $\M_{k}(B^n)$ is induced by the norm 
$|g|_{\C^\ell(B^n)}\coloneqq   \sup_{ij}|g_{ij}|_{\C^\ell(B^n)}$.
Here we record some basic observations on the structure of $\M_k(B^n)$, and its representation in normal coordinates. 

We say $B^n_g$ is \emph{strictly convex} provided that every pair of its points can be joined by a unique geodesic, and the second fundamental form of $\partial B^n_g$ with respect to the outward normal is positive definite. We need the
following fact whose proof utilizes the theory of  $\textup{CAT}(0)$ spaces \cites{bridson-haefliger1999,akp2024,bbi2001},
which are generalizations of Cartan-Hadamard manifolds. More precisely, a $\textup{CAT}(0)$ space is a metric space where every pair of points may be joined by a unique curve realizing the distance between the points, and the curvature is nonpositive in the sense of Alexandrov. 
  
 \begin{lem}\label{lem:convex}
 The set of metrics $g\in \M _{k}(B^n)$ such that $B^n_g$  is strictly convex is open in the $\C^1$-topology.
 \end{lem}
 \begin{proof}
 Fix a metric $g_0\in \M_{k}(B^n)$ such that $B^n_{g_0}$  is strictly convex.
 Consider metrics $g\in\M_{k}(B^n)$ with $|g-g_0|_{\C^1(B^n)}\leq \epsilon$, for a constant $\epsilon\in (0,\infty)$. The second fundamental form of $\partial B^n_g$, with respect to the outward normal $\nu$, is given by $\sff_g(v,w)\coloneqq g(D^g_v\nu,w)$, where $D^g$ is the covariant derivative with respect to $g$, and $v$, $w$ are tangent vectors of $\partial B^n_g$. Note that  $\partial B^n_g$ is a level set of the function $F(x)\coloneqq |x|$, where $|\cdot|$ is the Euclidean norm. Then $\nu(x)=\nabla^g F(x)/|\nabla^g F(x)|_g$, where $\nabla^g$ is the gradient with respect to $g$, and $|\cdot|_g\coloneqq \sqrt{g(\cdot,\cdot)}$. Furthermore, $D^g$ is determined by the Christoffel symbols, which depend on the first derivatives of $g$.  Hence $g\mapsto\sff_g$  is continuous in the $\C^1$-topology.  So, for sufficiently small $\epsilon$, $\sff_g$ remains positive definite, since $\sff_{g_0}$ is positive definite.  
 
 Since $\sff_g$ is positive definite and $g$ has nonpositive curvature, the curvature of $B^n_g$ is bounded above by $0$ in Alexandrov's sense \cite[p. 704]{alexander-berg-bishop1993}. So $B^n_g$ is locally a $\textup{CAT}(0)$ space. Then, since $B^n_g$ is simply connected, it is a $\textup{CAT}(0)$ space by the generalized Cartan-Hadamard theorem  \cites{alexander-bishop1990,bridson-haefliger1999,bbi2001}. In particular, every pair of points in $B^n_g$ may be joined by a unique shortest curve $\gamma$.
Since $\sff_g$ is positive definite, the interior of $\gamma$ cannot touch $\partial B^n_g$. Hence $\gamma$ is a (Riemannian) geodesic. So every pair of points of $B^n_g$ can be joined by a unique minimal geodesic. It follows that the cut locus of every point of $B^n_g$ is empty, so these geodesics are unique.
\end{proof}
 
 Let $\M^\star_{k}(B^n)\subset\M_{k}(B^n)$ consist of metrics such that $B^n_g$ is \emph{star-shaped} (with respect to its center $o$), i.e., there exists a domain $\ol{B^n_g}$ in the tangent space $T_o B^n_g\simeq\R^n$ such that the exponential map $\exp_o\colon \ol{B^n_g}\to B^n_g$ is a diffeomorphism, and the radial geodesics of $B^n_g$, which emanate from $o$, meet $\partial B^n_g$ transversely. In particular, if $B^n_g$ is strictly convex then it is star-shaped.
 
Since $g$ is symmetric and positive definite, there exists a positive definite symmetric matrix $h\coloneqq \sqrt{g_o}$.  Then $h^{-1}$ also exists and is positive definite. Let $\ol e_i\coloneqq   h^{-1}e_i$. Then 
$
g_o(\ol e_i, \ol e_j)=e_i^T (h^{-1})^Th^2 h^{-1}e_j=\delta(e_i,e_j)
$
by \eqref{eq:g}.
So $\ol e_i$ form an orthonormal basis for $T_oB^n_g$, which depends continuously on $g$.
Let $\phi=\phi_g\colon \R^n\to T_o B^n_g$ be the corresponding coordinate map given by $\phi(x)\coloneqq   \sum x_i\ol e_i$. Set $\Omega\coloneqq \phi^{-1}(\ol B^n_g)$, and $\tilde\exp_o\coloneqq \exp_o\!\circ\, \phi|_\Omega$. Then $\tilde\exp_o\colon \Omega\to B^n$ is a diffeomorphism. Let $\ol g\coloneqq   \tilde\exp_o^*(g)$ be the pullback metric. Then 
 $\tilde\exp_o\colon \Omega_{\ol g}\to B^n_g$ is an isometry. Furthermore,   the standard coordinates $x_i$ of $\R^n$  form normal coordinates on $\Omega_{\ol g}$. So
\begin{equation}\label{eq:g-delta}
\ol g_o(e_i,e_j)=g_o\big(d(\tilde\exp_o)_o(e_i), d(\tilde\exp_o)_o(e_j)\big)=g_o(\ol e_i, \ol e_j)=\delta(e_i,e_j).
\end{equation}

For $\theta\in S^{n-1}$, let $\rho_\theta$ be the radial geodesic which connects $o$ to $\partial \Omega_{\ol g}$ with initial direction $\theta$. Note that $\rho_\theta$ intersects $\partial \Omega_{\ol g}$ transversely, since the corresponding radial geodesics $\tilde{\exp_o}(\rho_\theta)$ of $B^n_g$ are transversal to $\partial B^n_g$ by definition. Let
$f_{\ol g}(\theta)\coloneqq   \textup{Length}(\rho_\theta)$ be the \emph{radial function} of $\partial\Omega_{\ol g}$.  Since $x_i$ are normal coordinates, $\rho_\theta(t)=t\theta$ for $0\leq t\leq f_{\ol g}(\theta)$.  So $\partial\Omega_{\ol g}$ is the graph of $f_{\ol g}$ over $S^{n-1}$ in the Euclidean sense. Note that  $\partial\Omega_{\ol g}$ is a $\C^\infty$ hypersurface, since $\partial\Omega_{\ol g}=(\tilde\exp_o)^{-1}(\partial B^n_g)$. Thus $f_{\ol g}$ is $\C^\infty$, since $\rho_\theta$ are transversal to $\partial\Omega_{\ol g}$.

 \begin{lem}\label{lem:continuous}
 The mapping $ \M  ^\star_{k}(B^n)\ni g\mapsto f_{\ol g}\in \C^1(S^{n-1})$ is continuous in the $\C^2$-topology. 
  \end{lem}
 \begin{proof}
 Fix $g\in  \M  ^\star_{k}(B^n)$ and let $g_i\in  \M  ^\star_{k}(B^n)$ be a sequence with $g_i\to g$ in the $\C^2$-topology.  Let $\exp_o$ and $\exp_o^i$ be the exponential maps of $B^n_g$ and $B^n_{g_i}$ respectively. Then $(\exp_o^i)^{-1}\to(\exp_o)^{-1}$ in the $\C^1$-topology of maps from $B^n$ to $T_o B^n\simeq\R^n$, because exponential maps are determined by ODE whose coefficients involve the first derivatives of the metric. Consequently,  
 $$
 \partial\Omega_{\ol{g}_{i}}=(\tilde{\exp_o^i})^{-1}(\partial B^n_{g_i})\;\;\longrightarrow\;\;(\tilde{\exp_o})^{-1}(\partial B^n_g)=\partial\Omega_{\ol g}
 $$ 
 in the $\C^1$-topology of hypersurfaces embedded in $\R^n$, because here the differentials 
 $d(\tilde{\exp_o^i})^{-1}=\phi_{g_i}^{-1}\circ d(\exp_o^i)^{-1}\to\phi^{-1}\circ d(\exp_o)^{-1}=d(\tilde{\exp_o})^{-1}$ in the $\C^0$-topology of maps from $B^n$ to $\R^n$, and so $(\tilde\exp_o^i)^{-1}\to \tilde\exp_o^{-1}$ in the $\C^{1}$-topology. Hence the radial functions
 $f_{\ol{g}_{i}}\to f_{\ol g}$ in $\C^1(S^{n-1})$.
\end{proof}

\section{Proof of the Main Result}
Here we establish the following result, which implies Theorem \ref{thm:main} via the above lemmas. For any metric $g\in {\mathcal{M}}_{k}^{\star }(B^{n})$, let $\Omega_{\ol g}\subset\R^n$ denote the corresponding star-shaped domain, in normal coordinates, with radial function $f_{\ol g}$ as defined above.  For $k\leq0$, let $\delta^k$ be the metric of constant curvature $k$ on $\R^n$ in normal coordinates centered at $o$. So $\delta^0=\delta$ is the standard Euclidean metric. Also note that $\Omega_{\delta^k}$ is isometric to a domain with radial function $f_{\ol g}$ in  $\R^n$ or hyperbolic space $\Hyp^n_k$, if $k=0$ or $k<0$ respectively.

\begin{thm}\label{thm:main2}
For every $R>0$ there exists $\epsilon >0$ such that for all metrics $g\in {\mathcal{M}}_{k}^{\star }(B^{n})$ 
with $|f_{\ol g}-R|_{\C^1(S^{n-1})}\leq \epsilon $, $\I(\Omega _{\ol{g}})\geq \I(\Omega_{\delta^k})$. 
Furthermore, $\I(\Omega_{\ol{g}})=\I(\Omega_{\delta^k})$ only if $\overline g=\delta^k$.
\end{thm}

If $\epsilon$ in Theorem \ref{thm:main} is sufficiently small, then $B^n_g$ is star-shaped by Lemma \ref{lem:convex}, so
$g\in {\mathcal{M}}_{0}^{\star }(B^{n})$. Furthermore, $|f_{\ol g}-1|_{\C^1(S^{n-1})}$ can be made arbitrarily small by Lemma \ref{lem:continuous}, since $f_{\ol{\delta}}=1$. So Theorem \ref{thm:main2}, together with the classical isoperimetric inequality, yields that 
$$
\I(B^n_g)=\I(\Omega_{\ol{g}})\geq \I(\Omega_\delta)\geq \I(B^n_\delta).
$$
If $\I(B^n_g)=\I(B^n_\delta)$, then $\I(\Omega_{\ol{g}})=\I(\Omega_\delta)$, which yields $\ol g=\delta$. So Theorem \ref{thm:main} indeed follows from Theorem \ref{thm:main2} (this argument cannot be generalized to $k<0$, as we had pointed out in Note \ref{note:H}, due to the use of the classical isoperimetric inequality).
Next, to prove Theorem \ref{thm:main2}, we begin by recording some basic facts from linear algebra. 
For any square matrix $A$, let $A(v,w)\coloneqq v^T \!\!A w$ denote the corresponding quadratic form.

\begin{lem}
\label{lem:LA} Let $A$ and $B$ be symmetric positive definite $n\times n$
matrices. Suppose that for all $v\in {\R}^{n}$, $A(v,v)\geq
B(v,v). $ Then

\begin{enumerate}
\item[(i)] {$\det(A)\geq \det(B)$},

\item[(ii)] {$\det(A) A^{-1}(v,v)\geq \det(B)B^{-1}(v,v)$}.
\end{enumerate}
Furthermore, equality holds in (i), and  in (ii) for all $v$, only if $A=B$.
\end{lem}

\begin{proof}
Let $Q^{-1}\Lambda Q$ be the spectral decomposition of $B$. Then 
$v^{T}Av\geq v^{T}Q^{-1}\Lambda Qv$. Setting $w\coloneqq \Lambda ^{1/2}Qv$, and 
$M\coloneqq \Lambda^{-1/2}QAQ^{-1}\Lambda ^{-1/2}$, we obtain 
$$
M(w,w)\geq |w|^2.
$$
So the eigenvalues $\lambda _{i}$ of $M$ are $\geq 1$. Thus $
\det (M)\geq 1$. But $\det (M)=\det (A)/\det (B)$. So we have (i). If
equality holds in (i), then $\lambda _{i}=1$, so $M$ is the identity matrix,
which yields $A=B$. Inequality (ii) is equivalent to 
$$
\det (M) M^{-1}(w,w)\geq |w|^2,
$$
which we rewrite as $\left( \prod \lambda _{i}\right) \sum w_{j}^{2}\lambda _{j}^{-1}\geq \sum w_{j}^{2}.$ This 
holds since $(\prod \lambda _{i})\lambda _{j}^{-1}\geq 1$, as $\lambda_{i}\geq 1$. So we obtain (ii). Equality holds in (ii) for all $v$ only if equality holds in the above inequality for all $w$. Then $\lambda_i=1$, which again yields $A=B$.
\end{proof}

Let $\ol{\exp}$ denote the exponential map of $\Omega _{\ol{g}}$. By \eqref{eq:g-delta}, the natural identification $
T_{o}\Omega _{\ol{g}}\simeq {\R}^{n}$ is an isometry. We also have $T_{x}(T_{o}\Omega _{\ol{g}})\simeq T_{x}{\R}^{n}\simeq {\R}^{n}$. Furthermore, since $x_{i}$ are normal coordinates, $\ol{\exp}_{o}(x)=x$ for $x\in \Omega _{\ol{g}}$. The volume element of $T_{x}\Omega _{\ol{g}}$ is $\sqrt{\det (\ol{g}_{x})}dx$, where $dx$ is the standard volume element of ${\R}^{n}$. Hence the Jacobian of $\ol{\exp}_{o}$ at $x$ is given by 
$$
J_{\overline g}(x)=\sqrt{\det (\ol{g}_{x})}.
$$
Set $J_k\coloneqq J_{\delta^k}$. Recall that $|\cdot|_g\coloneqq \sqrt{g(\cdot,\cdot)}$, and set 
$|\cdot|_k\coloneqq |\cdot|_{\delta^k}$. So $|\cdot|_0=|\cdot|$ is the Euclidean norm. Also recall that $\nabla^g$ denotes the gradient with respect to $g$, and set $\nabla^k\coloneqq \nabla^{\delta^k}$. So $\nabla^0=\nabla$ is the Euclidean gradient. Note that $g(v,\nabla^g f)=df(v)=\delta(v,\nabla f)$, which yields $v^T g \nabla^g f=v^T\nabla f$. So 
\begin{equation}\label{eq:nablag}
\nabla^g f=g^{-1}\nabla f.
\end{equation}
Let $r(x)\coloneqq|x|$ and $\theta (x) \coloneqq x/|x|$ be the polar coordinates on ${\R}^{n}\setminus \{o\}$. Note that $J_k(r\theta)$ does not depend on $\theta$. So we denote this quantity by $J_k(r)$. The last lemma together with Rauch's comparison theorem yields:

\begin{lem}
\label{lem:J} For any metric $g\in {\mathcal{M}}_{k}^{\star}(B^{n})$, and differentiable function $f$ on $\Omega_{\ol g}$,

\begin{enumerate}
\item[(i)] $r\mapsto J_{\ol{g}}(r\theta)/J_{k}(r)$ is nondecreasing,

\item[(ii)] $J_{\ol{g}}(r\theta)\geq J_{k}(r)$, with equality everywhere only if $\ol{g}=\delta^{k}$,

\item[(iii)] $J_{\ol{g}}(r\theta)\,|\nabla^{\ol{g}}f(r\theta)|_{\ol{g}}\geq J_{k}(r)\,|\nabla^{k}f(r\theta)|_{k}$.
\end{enumerate}
\end{lem}

\begin{proof}
For (i) see \cite[33.1.6]{burago-zalgaller1988} or \cite[p. 253]{bishop-crittenden1964}. Inequality (ii) follows from the inequality 
\begin{equation*}
\ol{g}(v,v)\geq \delta^k(v,v), 
\end{equation*}
which is due to Jacobi's equation in normal coordinates \cite[Thm. 11.10]{Lee2018}, and Lemma \ref{lem:LA}(i). Inequality (iii) also follows quickly from the above inequality via Lemma \ref{lem:LA}(ii) and \eqref{eq:nablag}.
\end{proof}

Now we are ready to establish our main result.

\begin{proof}[Proof of Theorem \protect\ref{thm:main2}]
Set $f=f_{\ol{g}}$,  and define
$$
\rho(\theta) \coloneqq \frac{J_{\ol{g}}\big(f(\theta)\theta\big)}{{J}_{k}\big(f(\theta)\big)},
\quad \quad\quad 
\mathcal{J}_{k}(s)\coloneqq\int_{0}^{s}J_{k}\big (r\big )r^{n-1}dr.
$$
By Lemma \ref{lem:J}(i), $J_{\ol g}(r\theta)\leq \rho(\theta)J_k(r)$ for $r\leq f(\theta)$.  Thus   the volume
\begin{multline*}
|\Omega _{\ol{g}}|
=
\int_{\Omega }J_{\ol{g}}(x)dx
=
\int_{S^{n-1}} \int_{0}^{f(\theta )}J_{\ol{g}}(r\theta )r^{n-1}drd\theta  \\
\leq 
\int_{S^{n-1}}\int_{0}^{f(\theta )}\rho(\theta) J_{k}\big (r \big )r^{n-1}drd\theta
= 
\int_{S^{n-1}}\rho(\theta)\mathcal{J}_{k}\big(f(\theta)\big)d\theta,
\end{multline*}
where $d\theta $ is the standard volume element of $S^{n-1}$.
Extend $f$ radially to $\Omega \setminus \{o\}$ by setting $f(x)\coloneqq f(\theta (x))$, and let $F(x)\coloneqq r(x)-f(x)$. Then $|\nabla^{\ol{g}}F|_{\ol{g}}^{2}=1+|\nabla ^{\ol{g}}f|_{\ol{g}}^{2}$
since $|\nabla^{\ol{g}}r|_{\ol{g}}=|\nabla r|=1$, and by the
Gauss lemma $\ol{g}(\nabla ^{\ol{g}}\,r,\nabla ^{\ol{g}}f)=0$. Thus, by the coarea formula, the perimeter
\begin{multline*}
\!\!\!\!|\partial \Omega _{\ol{g}}|
=
|F^{-1}(0)|
=
\frac{d}{ds}\Big |_{s=0}\int_{-s}^{0}\left|F^{-1}(t)\right|dt
=
\frac{d}{ds}\Big |_{s=0}\int_{F^{-1}([-s,0])}\left|\nabla^{\ol{g}}F(x)\right|_{\ol{g}}\,J_{\ol{g}}(x)dx \\
=
\frac{d}{ds}\Big |_{s=0}\int_{S^{n-1}}\int_{f(\theta )-s}^{f(\theta)}\left|\nabla^{\ol{g}}F(r\theta )\right|_{\ol{g}}\,J_{\ol{g}
}(r\theta )r^{n-1}drd\theta \\
=
\int_{S^{n-1}}\sqrt{1+|\nabla^{\ol{g}}f(\theta)|_{\ol{g}}^{2}}\,\,J_{\ol{g}}\big(f(\theta)\theta\big)f^{n-1}(\theta)\,d\theta .
\end{multline*}
Now applying Lemma \ref{lem:J}(iii), we obtain 
\begin{equation}\label{eq:main}
\I(\Omega _{\ol{g}})\geq \frac{\left( \int_{S^{n-1}}\sqrt{\rho^{2}+|\nabla ^{k}f|_{k}^{2}}\,\,J_{k}(f)f^{n-1}d\theta \right)^{n}}{\Big(\int_{S^{n-1}}\rho\, \mathcal{J} _{k}(f)d\theta \Big)^{n-1}}.
\end{equation}
So we may write 
$$
\I(\Omega _{\ol{g}})\geq \lambda^{n-1} (1),\quad\; \text{where}\quad\; 
\lambda (t)\coloneqq\frac{\left( \int_{S^{n-1}}\sqrt{\rho
^{2t}+|\nabla ^{k}f|_{k}^{2}}\,J_{k}(f)f^{n-1}d\theta \right)^{\frac{n}{n-1}}}{\int_{S^{n-1}}\rho ^{t}\mathcal{J} _{k}(f)d\theta },
$$
for $t\in \lbrack 0,1]$. Let $A^{\frac{n}{n-1}}$ denote the numerator and $B$
the denominator of $\lambda (t)$. Then 
\begin{equation}\label{eq:lambdaprime}
\lambda ^{\prime }(t)=\frac{A^{\frac{1}{n-1}}}{B}\left( \frac{n}{n-1}
A^{\prime }-\frac{A}{B}B^{\prime }\right) =\frac{A^{\frac{1}{n-1}}}{B}
\int_{S^{n-1}}C\ln (\rho )\rho ^{t}J_{k}(f)f^{n-1}\,d\theta ,
\end{equation}
where 
$$
C\coloneqq\frac{n}{n-1}\frac{\rho ^{t}}{\sqrt{\rho ^{2t}+\left|\nabla^{k}f\right|_{k}^{2}}}-\frac{A}{B}\frac{\mathcal{J} _{k}(f)}{\,J_{k}(f)f^{n-1}}.
$$
By Lemma \ref{lem:J}(ii), $\rho \geq 1$. Thus the sign of $\lambda'$ depends on that of $C$. By assumption, 
$|f-R|\leq \epsilon $ and $|\nabla f|\leq \epsilon $. So by \eqref{eq:nablag} and continuity of $J_k$ and $\mathcal{J}_k$,  for any $\overline{\epsilon }>0$ we may choose $\epsilon $ so small that 
$$
|f-R|\leq\ol\epsilon, \quad\;\;
|\nabla ^{k}f|_{k}\leq\ol\epsilon,\quad\;\;
\left|\frac{\,J_{k}(f)}{J_{k}\big (R\big )}-1\right|\leq\ol\epsilon,
\quad\text{and}\quad
\left|\,\frac{\mathcal{J} _{k}(f)}{\mathcal{J} _{k}(R)}-1\right| \leq\ol\epsilon.
$$
Then 
\begin{gather*}
A
\leq 
\sqrt{1+\overline{\epsilon }^{2}}J_{k}(R) \left( 1+\overline{\epsilon }\right) (R+\overline{\epsilon })^{n-1}\int_{S^{n-1}}\rho ^{t}d\theta,
\\
B
\geq 
\mathcal{J} _{k}(R)(1-\overline{\epsilon })\int_{S^{n-1}}\rho
^{t}d\theta .
\end{gather*}
Thus 
$$
C
\geq 
\frac{n}{n-1}\frac{1}{\sqrt{1+\overline{\epsilon }^{2}}}
-
\left(\frac{1+\ol\epsilon}{1-\ol\epsilon}\right)^2\left(\frac{R+\overline{\epsilon }}{R-\overline{\epsilon }}\right)^{n-1}\!\sqrt{1+\overline{\epsilon }^{2}}.
$$
So if $\ol\epsilon$ is sufficiently small, then $C>0$, which yields $\lambda ^{\prime }\geq 0$. Thus $\lambda (1)\geq \lambda (0)$. 
But $\lambda^{n-1}(0)=\I(\Omega _{\delta^k})$. So $\I(\Omega _{\ol{g}})\geq \I(\Omega _{\delta^k})$ as desired.

Finally suppose that $\I(\Omega _{\ol{g}})= \I(\Omega _{\delta^k})$. Then $\lambda(1)=\lambda(0)$. So $\lambda^{\prime}=0$ identically. Since $C>0$, it follows from \eqref{eq:lambdaprime} that $\rho=1$ identically. So $J_{\overline g}(f(\theta)\theta)=J_{k}(f(\theta))$. Consequently, by Lemma \ref{lem:J}(i), $J_{\overline g}(r\theta)\leq J_{k}(r)$ for $r\leq f(\theta)$. Thus $\overline g=\delta^k$ by Lemma \ref{lem:J}(ii).
\end{proof}

\section*{Acknowledgments}
We thank Igor Belegradek, Joe Hoisington, Michael Loss, and Erik I. Verriest for helpful discussions. Thanks also to the anonymous referees for suggestions to enhance the exposition of this work.

\bibliographystyle{amsplain}
\bibliography{references}

\end{document}